\documentclass[fleqn,11pt]{article}
\usepackage{amsmath}
\usepackage{latexsym}
\usepackage{rotate,epsfig}
\usepackage{lscape}
\numberwithin{equation}{section}

\newtheorem{theorem}{\bf Theorem}[section]
\newtheorem{proposition}[theorem]{\bf Proposition}
\newtheorem{lemma}[theorem]{\bf Lemma}
\newtheorem{corollary}[theorem]{\bf Corollary}

\newenvironment{remark}%
    {\par \noindent {\bf Remark }}%
    {\par \indent}
\newenvironment{proof}%
    {\par \noindent {\bf Proof}}%
    {\par \indent}
\begin{document}
\title{{\bf Pareto analysis based on records}}
\author{{\bf M. Doostparast}${}^{1,}$\footnote{Corresponding author.\newline
\textit{E-mail addresses:} doostparast@math.um.ac.ir (M.
Doostparast), bala@mcmaster.ca (N. Balakrishnan). }
\ and \ {\bf N. Balakrishnan}${}^{2}$
   \\
{\small {\it $~^{1}$Department of Statistics, School of
Mathematical Sciences,}}\vspace{-0.2cm}\\ {\small {\it Ferdowsi University of Mashhad, P. O. Box 91775-1159,  Mashhad,\ Iran }}\\
{\small {\it $~^{2}$Department of Mathematics and Statistics,
McMaster University,
 }}\vspace{-0.2cm}\\
{\small {\it Hamilton,
Ontario, Canada L8S 4K1}}\vspace{-0.2cm}\\
}
\date{}
\maketitle
\begin{abstract}
Estimation of the parameters of an exponential distribution based
on record data has been treated by Samaniego and Whitaker (1986)
and Doostparast (2009). Recently, Doostparast and Balakrishnan
(2011) obtained optimal confidence intervals as well as uniformly
most powerful tests for one- and two-sided hypotheses concerning
location and scale parameters based on record data from a
two-parameter exponential model. In this paper, we derive optimal
statistical procedures including point and interval estimation as
well as most powerful tests based on record data from a
two-parameter Pareto model. For illustrative purpose, a data set
on annual wages of a sample of production-line workers
in a large industrial firm is analyzed using the proposed
procedures.
\end{abstract}
\section{Introduction}\label{intro}
Let $X_1,X_2,X_3,\cdots$ be a sequence of continuous random
variables. $X_k$ is a lower record value if it is smaller than
all preceding values $X_1,X_2,\cdots,X_{k-1}$ and by definition,
$X_1$ is taken as the first lower record value. An analogous
definition can be provided for upper record values. {\bf
Such data may be represented by
$(\textbf{R,K}):=(R_1,K_1,\cdots,R_m,K_m)$, where $R_i$ is
the $i$-th record value meaning new minimum (or maximum) and
$K_i$ is the number of trials following the observation of $R_i$
that are needed to obtain a new record value $R_{i+1}$. Throughout this paper, we denote the observed value of these record data by $(\textbf{r,k}):=(r_1,k_1,\cdots,r_m,k_m)$.}
Record
statistics arise naturally in many practical problems and in
applied fields such as athletic events (Kuper and Sterken,
2003), Biology (Krug and Jain, 2005), catastrophic loss (Hsieh,
2004 and Pfeifer, 1997), climate research (Benestad, 2003),
financial markets (Bradlow and Park, 2007 and de Haan {\em et al.},
2009), industrial application (Samaniego and Whitaker, 1986 and
1988), spatial patterns (Yang and Lee, 2007), and  traffic analysis
(Glick, 1978).  Hence, finding optimal statistical inferential
procedures based on record data becomes very important and useful from a data-analysis point of view.

The rest of this article is organized as
follows.  In Section 2, we present briefly the notation to be used through out
the paper and also the form of Pareto distribution to be studied
here. In Section 3, we describe the basic form of record data to be
considered and the corresponding likelihood function. In Section 4,
we discuss the optimal point estimation of the Pareto parameters,
while the interval estimation is handled in Section 5. Tests of
hypotheses concerning the parameters are discussed in Section 6
and finally a numerical example is presented in Section 7 in order to
illustrate all the inferential procedures developed here.

\section{Some Preliminaries }
Throughout the paper, we will use the following notation:\\~\\
\begin{tabular}{lcl}\hline
$Exp(\mu,\sigma)$ &:& Exponential distribution with location $\mu$ and scale $\sigma$\\
$Gamma(n,\sigma)$ &:& Gamma distribution with shape $n$ and scale $\sigma$\\
$Par(\beta,\alpha)$ &:& Pareto distribution with scale $\beta$ and shape $\alpha$\\
$\chi^2_v$ &:& Chi-square distribution with $v$ degrees of freedom\\
$\chi^2_{v,p}$ &:& $100\gamma^{th}$ percentile of the chi-square distribution with $v$ degrees of freedom\\
$({\bf r,k})$ &:& $(r_1,k_1, \cdots ,r_m,k_m)$\\
$({\bf R,K})$ &:& $(R_1,K_1, \cdots ,R_m,K_m)$\\
$T_m$ &:& $\sum_{i=1}^m K_i$, the time of occurrence of the $m$-th record \\
$T_1^\star$ &:& $\sum_{i=1}^m K_i(\log R_i-\log\beta)$\\
$T_2^\star$ &:& $\sum_{i=1}^{m-1} K_i(\log R_i-\log R_m)$\\
$\Gamma(r)$&:& $\int_0^{\infty} x^{r-1} e^{-x}dx$, the complete gamma function\\
$\hat{\theta}_M$ &:& Maximum likelihood estimator of $\theta$\\
$\hat{\theta}_U$ &:& Unbiased estimator of $\theta$\\ \hline
\end{tabular}
~~\\~~\\

A random variable $X$ is said to have a Pareto distribution,
denoted by $X\sim Par(\beta,\alpha)$, if its cumulative
distribution function (cdf) is
\begin{equation}\label{cdf:par}
F(x;\beta,\alpha)=1-\left(\frac{\beta}{x}\right)^\alpha,\ \ \ x\geq\beta>0,\ \alpha>0,
\end{equation}
and the probability density function (pdf) is
\begin{equation}\label{pdf:par}
f(x;\beta,\alpha)=\alpha \beta^\alpha x^{-(\alpha+1)},\ \ \ x\geq\beta>0,\ \alpha>0.
\end{equation}
For a through discussion on various properties and applications
and different forms of Pareto distribution, one may refer to
Arnold (1983) and Johnson, Kotz and Balakrishnan (1994).

\section{Form of Data}
As in Samaniego and Whitaker (1986) and
Doostparast (2009), our starting point is a sequence of
independent random variables $X_1,X_2,X_3,\cdots$ drawn from a
fixed cdf $F(\cdot)$ and pdf $f(\cdot)$. We
assume that only successive minima are observable, so that the
data may be represented as $({\bf
r,k}):=(r_1,k_1,r_2,k_2,\cdots,r_m,k_m)$, where $r_i$ is the
value of the $i$-th observed minimum, and $k_i$ is the number of
trials required to obtain the next new minimum. The likelihood function associated with the
sequence $\{r_1,k_1,\cdots,r_m,k_m\}$ is given by
\begin{equation}\label{like}
L({\bf r,k})=\prod_{i=1}^m f(r_i)
[1-F(r_i)]^{k_i-1}I_{(-\infty,r_{i-1})},
\end{equation}
where $r_0\equiv \infty$, $k_m\equiv 1$, and $I_A(x)$ is the
indicator function of the set $A$.\\

The above described scheme is known as {\em inverse sampling
scheme}.  Under this scheme, items are presented sequentially and
sampling is terminated when the $m$-th minimum is observed. In
this case, the total number of items sampled is a random number,
and $K_m$ is defined to be one for convenience. There is yet
another common scheme called {\em random sampling scheme} that is
discussed in the literature. Under this scheme, a random sample
$Y_1,\cdots,Y_n$ is examined sequentially and successive minimum
values are recorded.  In this setting, we have $N^{(n)}$, the
number of records obtained, to be
random and, given a value of $m$, we have in this case $\sum_{i=1}^m K_i=n$. \\

\begin{remark}
Doostparast and Balakrishnan (2011) derived classical estimators
for $Exp(\theta,\sigma)$-model under both inverse and random sampling
schemes,
and also discussed associated cost-benefit analysis.
\end{remark}
\section{Point Estimation}
Let us now assume that the sequence $\{R_1,K_1,\cdots,R_m,K_m \equiv 1\}$ is arising from
$Par(\beta,\alpha)$ in (\ref{cdf:par}).  Then, the
likelihood function in (\ref{like}) becomes
    \begin{equation}\label{like:par}
        L(\beta,\alpha;{\bf r,k})=\frac{\alpha^m \beta^{\alpha\sum_{i=1}^m k_i}}{\prod _{i=1}^m r_i^{\alpha k_i+1}},\ \ \ 0<\beta\leq r_m,\ \
        \alpha>0,
    \end{equation}
and so the log-likelihood function is
    \begin{equation}\label{like:log:par}
        l(\beta,\alpha;{\bf r,k})=m\ln\alpha-\alpha\sum_{i=1}^m k_i(\ln r_i-\ln\beta)-\sum_{i=1}^m \ln r_i,\ \ \ 0<\beta\leq r_m,\ \ \alpha>0.
    \end{equation}
Since $\frac{\partial}{\partial \beta}l(\beta,\alpha;{\bf r,k})=\alpha\beta^{-1}\sum_{i=1}^m k_i>0$,
$l(\beta,\alpha;{\bf r,k})$ is increasing  with respect to $\beta$. This implies that
\begin{equation}\label{ml:beta}
\hat{\beta}_M=R_m.
\end{equation}
Substituting \eqref{ml:beta} in \eqref{like:log:par}, the maximum
likelihood estimate of $\alpha$ is readily obtained as
\begin{equation}\label{ml:alpha}
\hat{\alpha}_M=\frac{m}{T_2^\star}.
\end{equation}
Furthermore, $(\sum_{i=1}^m K_i \ln R_i, T_m,  R_m)$ is a joint
sufficient statistic for ($\beta$, $\alpha$).
\begin{corollary}\label{col:statis:gamma}
It can be shown that $T_1^\star$ and $T_2^\star$ are distributed
as $Gamma(m,\alpha^{-1})$ and $Gamma(m-1,\alpha^{-1})$,
respectively.
\end{corollary}
From Corollary \ref{col:statis:gamma}, an unbiased estimator for
$\alpha$ is given by
\[\hat{\alpha}_U=\frac{m-1}{T_2^\star}.\]
In the following, we show that $\hat{\alpha}_M$ dominates
$\hat{\alpha}_U$, under square error (SE) loss  function. In
other words, $\hat{\alpha}_U$ is inadmissible under SE loss
function. First, we need the following lemma.
\begin{lemma}\label{gamma:moments}
Suppose $X$ has a $Gamma(\nu,\tau)$ distribution. Then,
\[E(X^{-k})=\tau^{-k}\frac{\Gamma(\nu-k)}{\Gamma(\nu)},\ \ \ k<\nu.\]
\end{lemma}
\begin{proposition}
For $m>4$, under the SE loss  function, $\hat{\alpha}_M$ dominates
$\hat{\alpha}_U$.
\end{proposition}
\begin{proof}
From Lemma \ref{gamma:moments}, we have
\begin{eqnarray}
MSE(\hat{\alpha}_M)&:=&E\left(\hat{\alpha}_M-\alpha\right)^2\nonumber\\
&=&E\left(\frac{m}{T_2^\star}-\alpha\right)^2\nonumber\\
&=&m^2 \alpha^2\frac{\Gamma(m-1-2)}{\Gamma(m-1)}+\alpha^2-2\alpha^2 m \frac{\Gamma(m-1-1)}{\Gamma(m-1)}\nonumber\\
&=&\alpha^2\left\{\frac{m^2}{(m-2)(m-3)}+1-\frac{2m}{m-2}\right\}\nonumber\\
&=&\alpha^2\frac{(m+6)}{(m-2)(m-3)}.\label{mse:ml:alpha1}
\end{eqnarray}
Replacing $m$ with $m-1$ in \eqref{mse:ml:alpha1}, we immediately have
\begin{equation}\label{mse:ml:alpha}
MSE(\hat{\alpha}_U)=\alpha^2\frac{(m+5)}{(m-3)(m-4)}.
\end{equation}
Thus, the efficiency of $\hat{\alpha}_M$ with respect to
$\hat{\alpha}_U$ is given by
\begin{eqnarray}
EFF(\hat{\alpha}_M,\hat{\alpha}_U)&=&\frac{MSE(\hat{\alpha}_U)}{MSE(\hat{\alpha}_M)}\nonumber\\
&=&\frac{(m+5)}{(m-3)(m-4)}\frac{(m-2)(m-3)}{(m+6)}\nonumber\\
&=&\frac{(m+5)(m-2)}{(m+6)(m-4)}\nonumber\\
&=&1+\frac{m+14}{(m+6)(m-4)}\label{eff:ml:to:un:alpha}\\
&>&1,\nonumber
\end{eqnarray}
which is the desired result. \hfill{$\Box$}
\end{proof}

One can easily check that the bias of $\hat{\alpha}_M$, under the SE
loss function, is
\[B_{SE}(\hat{\alpha}_M,\alpha)=E(\hat{\alpha}_M)-\alpha=2\alpha/(m-2).\]
It may be noted that
\[\lim_{m\to \infty}EFF(\hat{\alpha}_M,\hat{\alpha}_U)=1.\]
This is to be expected since $\hat{\alpha}_M$ and
$\hat{\alpha}_U$ are equivalent for large values of $m$. Figure
\ref{fig:eff:mle:to:unbiase} shows the relative efficiency of
$\hat{\alpha}_M$ with respect to $\hat{\alpha}_U$. \begin{figure}
\begin{center}
\centering
\includegraphics[angle=0,scale=0.5]{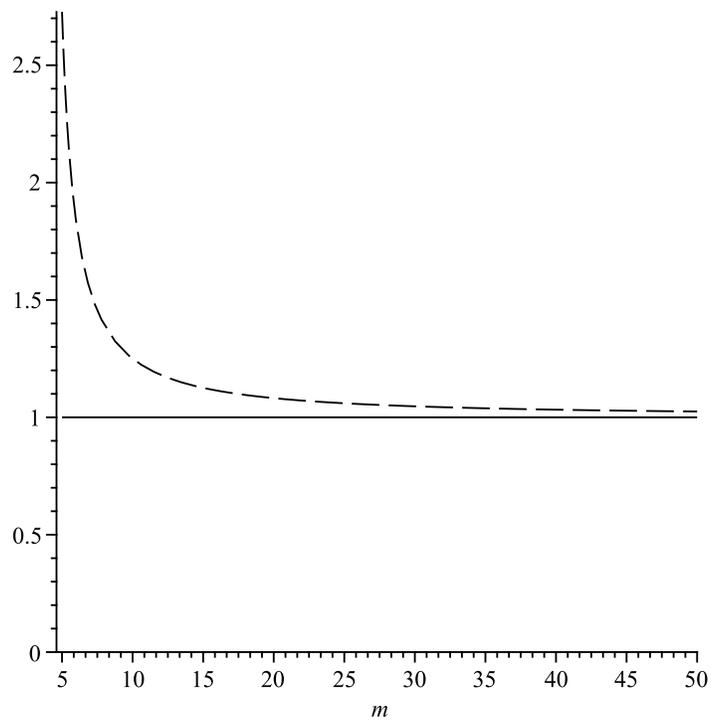}
\caption{Relative efficiency of $\hat{\alpha}_M$ with respect to $\hat{\alpha}_U$ given by \eqref{eff:ml:to:un:alpha}. }\label{fig:eff:mle:to:unbiase}
\end{center}
\end{figure}
\section{Confidence intervals}
Suppose we observe $({\bf r,k})$ from a two-parameter Pareto
distribution in \eqref{cdf:par}. In this section, we discuss the
construction of exact confidence intervals for the two parameters
in different cases.

\subsection{$\alpha$ known}
Suppose the shape parameter $\alpha$ is known. Then, from
\eqref{like:par}, we have $(T_m,R_m)$ to be a joint sufficient
statistic for $\beta$. Since $T_m$ is distributed free from
parent distribution (Glick, 1978), we consider two approaches for
obtaining confidence intervals for $\beta$ on the basis of record
data.
\subsubsection*{{\em Unconditional method}}
To obtain a confidence interval for $\beta$, we need the
following lemma due to Doostparast and Balakrishnan (2011).
\begin{lemma}\label{doost:bala:lemma}
Suppose we observe $({\bf r,k})$ from a two-parameter
 $Exp(\mu,\sigma)$-distribution. Then
\begin{equation}\label{pdf:rm:expo}
\frac{R_m-\mu}{\sigma}\sim
g(x;m)=\frac{\left\{-\ln(1-\exp(-x))\right\}^{m-1}}{\Gamma(m)}\exp(-x),\
\ \ x\geq 0.
\end{equation}
\end{lemma}
\begin{lemma}\label{par:to:expo:lemma}
If $X\sim Par(\beta,\alpha)$, then $\ln X\sim Exp(\ln\beta,\alpha^{-1})$.
\end{lemma}
From Lemmas \ref{doost:bala:lemma} and \ref{par:to:expo:lemma},
it can be shown in this case that
\begin{equation}\label{pdf:rm:alpha:known}
\alpha\left(\ln R_m-\ln\beta\right)\sim
g(x;m).
\end{equation}
This implies that $-\ln\left(1-\exp\left\{-\alpha\left(\ln
R_m-\ln\beta\right)\right\}\right)$ has a gamma distribution with
parameters ($m$,$1$), and therefore,
\begin{equation}\label{pdf:rm:alpha:known}
-2\ln\left(1-\left(\frac{\beta}{R_m}\right)^\alpha\right)\sim\chi^2_{(2m)}.
\end{equation}
Hence, an equi-tailed $100(1-\gamma)\%$ confidence interval for
$\beta$ is given by
\begin{equation}\label{conf:beta:et:1}
I_{ET,1}(\beta)=\left(R_m\sqrt[\alpha]{1-\exp\left\{-\frac{\chi^2_{2m,\frac{\gamma}{2}}}{2}\right\}},
R_m\sqrt[\alpha]{1-\exp\left\{-\frac{\chi^2_{2m,1-\frac{\gamma}{2}}}{2}\right\}}\right).
\end{equation}

Suppose we restrict our attention to intervals of the form
$(aR_m,bR_m)$, where $0<a<b$. Since the function $g(x;m)$ in
\eqref{pdf:rm:expo} is decreasing with respect to $x$ for every
$m\geq 1$, the $100(1-\gamma)\%$ confidence interval for $\beta$
with minimum width in this subclass of intervals is
\begin{equation}
I_{ML,1}(\beta)=\left(R_m\exp\left\{-\frac{g_{m,1-\gamma}}{\alpha}\right\},R_m\right),
\end{equation}
where $g_{m,\gamma}$ is $100\gamma$-th percentile of the pdf
$g(x;m)$ in \eqref{pdf:rm:expo}.

\subsubsection*{{\em Conditional method}}
Since minimum of a random sample of size $n$ from $Exp(0,\sigma)$
has a $Exp(0,\sigma/n)$-distribution [see Arnold, Balakrishnan and Nagaraja (1992)], conditional on $T_m=j$ for
$j\geq m$, the random variable $\displaystyle \alpha j(\ln
R_m-\ln\beta)$ has a standard exponential distribution.
Therefore, a conditional equi-tailed $100(1-\gamma)\%$ confidence
interval for $\beta$ is given by
\[\left(R_m\sqrt[\alpha j]{\frac{\gamma}{2}},R_m\sqrt[\alpha j]{1-\frac{\gamma}{2}}\right).\]
This
implies that an equi-tailed  $100(1-\gamma)\%$ confidence interval
for $\beta$ is
\begin{equation}\label{conf:beta:et:cond:1}
I_{ET,1,C}(\beta)=\left(R_m\left(\frac{\gamma}{2}\right)^{\frac{1}{\alpha
T_m}},R_m\left(1-\frac{\gamma}{2}\right)^{\frac{1}{\alpha
T_m}}\right).
\end{equation}
The expected width of the interval in \eqref{conf:beta:et:1} is
\begin{equation}
L(I_{ET,1}(\beta))=E(R_m)\left(\sqrt[\alpha]{1-\exp\left\{-\frac{\chi^2_{2m,1-\frac{\gamma}{2}}}{2}\right\}}
-\sqrt[\alpha]{1-\exp\left\{-\frac{\chi^2_{2m,\frac{\gamma}{2}}}{2}\right\}}\right),
\end{equation}
while the expected width of the interval in \eqref{conf:beta:et:cond:1} is
\begin{eqnarray*}
L(I_{ET,C,1}(\beta))&=&E\left(R_m\left\{\left(1-\frac{\gamma}{2}\right)^{\frac{1}{\alpha T_m}}-\left(\frac{\gamma}{2}\right)^{\frac{1}{\alpha T_m}}\right\}\right)\\
&=&\sum_{j=m}^{\infty}E\left(R_m\left\{\left(1-\frac{\gamma}{2}\right)^{\frac{1}{\alpha T_m}}-\left(\frac{\gamma}{2}\right)^{\frac{1}{\alpha T_m}}\right\}|T_m=j\right)P(T_m=j)\\
&=&\sum_{j=m}^{\infty}\left\{\left(1-\frac{\gamma}{2}\right)^{\frac{1}{\alpha j}}-\left(\frac{\gamma}{2}\right)^{\frac{1}{\alpha j}}\right\}E\left(R_m|T_m=j\right)P(T_m=j).
\end{eqnarray*}
Again, since the minimum of a random sample of size $n$ from
$Exp(0,\sigma)$ has a $Exp(0,\sigma/n)$-distribution, from Lemma
\ref{par:to:expo:lemma}, we conclude that $\ln
R_m|\left(T_m=j\right)$ has a $Exp(\ln\beta,(\alpha j)^{-1})$
distribution. So,
\begin{eqnarray*}
E\left(R_m|T_m=j\right)&=&E\left(\exp\{\ln R_m\}|T_m=j\right)\\
&=&\int_{\ln\beta}^{\infty}e^y j\alpha e^{-j\alpha(y-\ln\beta)}dy\\
&=&\frac{\beta j\alpha}{j\alpha-1}.
\end{eqnarray*}
Therefore,
\begin{eqnarray*}
L(I_{ET,C,1}(\beta))&=&\sum_{j=m}^{\infty}\left\{\left(1-\frac{\gamma}{2}\right)^{\frac{1}{\alpha j}}-\left(\frac{\gamma}{2}\right)^{\frac{1}{\alpha j}}\right\}\frac{\beta j\alpha}{j\alpha-1}P(T_m=j)\\
&=&E\left(\left\{\left(1-\frac{\gamma}{2}\right)^{\frac{1}{\alpha T_m}}-\left(\frac{\gamma}{2}\right)^{\frac{1}{\alpha T_m}}\right\}\frac{\beta T_m\alpha}{T_m\alpha-1}\right)\\
&=&\beta E\left(\left\{\left(1-\frac{\gamma}{2}\right)^{\frac{1}{\alpha T_m}}-\left(\frac{\gamma}{2}\right)^{\frac{1}{\alpha T_m}}\right\}\frac{ T_m\alpha}{T_m\alpha-1}\right).
\end{eqnarray*}
Hence, for computing the expected width of the interval in
\eqref{conf:beta:et:cond:1}, we need the probability mass
function of $T_m$. From Sibuya and Nishimura (1997), we have
\begin{equation}\label{pmf:tm}
P(T_m=j)=\frac{1}{j!}\left[\begin{array}{c}j-1\\m-1\end{array}\right],\
\ \ j\geq m,
\end{equation}
where brackets $[\ ]$ denote unsigned Stirling numbers of the first kind defined by the polynomial identity
\[z^{[n]}:=z(z+1)\cdots (z+n-1)=\sum_{m=1}^n \left[\begin{array}{c}n\\m\end{array}\right]z^m.\]
Now, let $H(\cdot)$ be an arbitrary function. Then, Doostparast
and Balakrishnan (2009) showed that
\begin{equation}\label{expec:h:tm}
E(H(T_m))=E\left(\frac{T_m H(T_m-1)}{T_m-2}\right)- E\left(\frac{ H(T_{m-1})}{T_{m-1}-1}\right)
\end{equation}
and this formula may be used for obtaining the required
expectations by taking a suitable choice for the function
$H(\cdot)$. However, no explicit expression seems possible for
$E\left(a^{\frac{1}{\alpha T_m}}\frac{
T_m\alpha}{T_m\alpha-1}\right)$ and so a simulation study was
carried out to generate sequences of independent observations based
on which the desired estimates were calculated
for $E\left(a^{\frac{1}{\alpha T_m}}\frac{ T_m\alpha}{T_m\alpha-1}\right)$ in the illustrative examples.\\

\begin{remark}
The theory of uniformly most powerful (UMP) one-sided test can be
applied to the problem of obtaining a lower or upper bounds. In
Section \ref{sec:test}, we will obtain uniformly most accurate
(UMA) lower and upper bounds for $\beta$.
\end{remark}

\subsection{$\beta$ known}
If $\beta$ is known, then $T_1^\star$ is a complete sufficient
statistic for $\alpha$, and so confidence intervals can be based
on this statistic. Since $T_1^\star$ is distributed as
$Gamma(m,\alpha^{-1})$, we have
\begin{equation}\label{pdf:T1:star}
2\alpha T_1^\star\sim\chi^2_{(2m)}.
\end{equation}
Therefore, in practice, one may use equi-tailed $100(1-\gamma)\%$ interval of the form
\[I_{ET,1}(\alpha):=\left(\frac{\chi^2_{2m,\gamma/2}}{2 T_1^\star},\frac{\chi^2_{2m,1-(\gamma/2)}}{2 T_1^\star}\right).\]
Suppose we restrict ourselves to a class of intervals of the form
\begin{equation}\label{class:int:alpha}
I_1(a,b)=\left(\frac{a}{2 T_1^\star},\frac{b}{2 T_1^\star}\right),\ \ \ \ 0<a<b.
\end{equation}
We then need to find $a$ and $b$ that minimizes
the width  of the interval in \eqref{class:int:alpha} subject to the confidence coefficient being
$1-\gamma$. Using Lagrange method, we then need to solve the following equations for
$a$ and $b$, determining $I_{ML}(\alpha)$, as
\begin{equation}\label{IML:beta:known}
\int_a^b h_{2m}(x)dx=1-\gamma\ \ \ \mathrm{and}\ \ \ h_{2m}(a)=h_{2m}(b),
\end{equation}
where $h_v(x)$ is the density function of a chi-square distribution
with $v$ degrees of freedom given by
\begin{equation}\label{pdf:chi}
h_v(x)=\frac{1}{2^{v/2}\Gamma(\frac{v}{2})}x^{v/2-1}\exp\left(-\frac{x}{2}\right),\
\ \ x>0.
\end{equation}
Table \ref{values:a:b:ML:beta:known} presents values of $a$ and $b$
up to six decimal places that satisfy the conditions in \eqref{IML:beta:known}.
\begin{table}
  \centering
  \begin{tabular}{c|cccccc}\hline
&&&$m$&&&\\
$\gamma$ &2&3&4&5&6&7\\ \hline
0.10 &0.167630&0.882654&1.874590&3.017327&4.258219&5.569586\\
     &7.864292&10.958349&13.892227&16.710795&19.446252&22.118958\\
0.05 &0.084727&0.607001&1.425002&2.413920&3.516159&4.700465\\
     &9.530336&12.802444&15.896592&18.860434&21.728898&24.524694\\
0.01 &0.017469&0.263963&0.785646&1.497847&2.344412&3.291176\\
     &13.285448&16.901320&20.295553&23.532765&26.653130&29.683220\\
   \hline
  \end{tabular}
  \caption{Values of $a$ (the upper figure) and $b$ (the lower figure) in \eqref{IML:beta:known} for $\gamma=0.01,0.05, 0.1$
  and different choices of $m$.}\label{values:a:b:ML:beta:known}
\end{table}
\\~~

Suppose that the random variable $X$ has a $Gamma(v,\tau)$-distribution, where $v$ is a known constant.
A UMP test does not exist for testing $H_0:\tau=\tau_0$ against the alternative $H_1:\tau\neq\tau_0$
(Lehmann, 2000, p. 111). So, there are no UMA bounds for $\alpha$. However, the acceptance region of the
UMP unbiased test is
\[C_1\leq \frac{2X}{\tau_0}\leq C_2,\]
where $C_1$ and $C_2$ are obtained from the equations
\[\int_{C_1}^{C_2}h_{2v}(x)dx=1-\gamma\ \ \ \mbox{and}\ \ \ C_1^v e^{-C_1/2}=C_2^v e^{-C_2/2}.\]
This yields UMA unbiased bounds for $\alpha$ as
\begin{equation}\label{uma:unbiased:alpha}
\left(\frac{C_1}{2 T_1^{\star}},\frac{C_2}{2 T_1^{\star}}\right),
\end{equation}
with
\[\int_{C_1}^{C_2}h_{2m}(x)dx=1-\gamma\ \ \ \mbox{and}\ \ \ C_1^m e^{-C_1/2}=C_2^m e^{-C_2/2}.\]
\begin{corollary}
UMA unbiased and minimum width intervals in the class
\eqref{class:int:alpha} given by \eqref{uma:unbiased:alpha} and
\eqref{IML:beta:known}, respectively, are identical.
\end{corollary}

\begin{remark}
From Lehmann (2005, p. 72, Theorem 3.5.1) and \eqref{pdf:T1:star}, the acceptance region of the most powerful test of $H_0:\alpha=\alpha_0$ against $H_1:\alpha<\alpha_0$ is $2\alpha_0 T_1^{\star}\leq C_u$, where $C_u$ is determined by the equation
\[\int_0^{C_U}h_{2m}(x)dx=1-\gamma.\]
Therefore, $\displaystyle\frac{\chi^2_{2m,1-\gamma}}{2
T_1^{\star}}$ is a UMA upper confidence bound for $\alpha$.
Similarly, $\displaystyle\frac{C_L}{2 T_1^{\star}}$ is a UMA lower
confidence bound for $\alpha$, where $C_L$ is such that
\[\int_{C_L}^{\infty}h_{2m}(x)dx=1-\gamma,\]
or
\[\int_{0}^{C_L}h_{2m}(x)dx=\gamma.\]
That is, $\displaystyle\frac{\chi^2_{2m,\gamma}}{2 T_1^{\star}}$ is
a uniformly most accurate lower confidence bound for $\alpha$
(without the restriction of unbiasedness). For more details, one may refer to
Lehmann (2005) and Pachares (1961) for tables of $C_1$ and
$C_2$.
\end{remark}

\subsection{$\beta$ and $\alpha$ both unknown}
From  \eqref{like:par}, the statistic $(T_2^\star, T_m, R_m)$ is
jointly sufficient
 for $\beta$ and $\alpha$.  Therefore, confidence intervals
 may be developed based on this statistic.
\subsubsection*{Confidence interval for $\alpha$}
From Lemma \ref{par:to:expo:lemma}, $\alpha^{-1}$ is a scale
parameter for data $(R_1',K_1,\cdots,R_m',K_m)$, where $R_i'=\ln
R_i$ for $1\leq i\leq m$. Therefore, the assumption that the
limits remain unchanged upon the addition of a constant to all
log-record values ($R_i'$) seems reasonable and leads to
intervals depending only on $T_2^\star$. For convenience, we
restrict ourselves to multiples of $T_2^\star$ for intervals of
the form
\begin{equation}\label{class2:int:alpha}
I_2(a,b)=\left(\frac{a}{2 T_2^\star},\frac{b}{2
T_2^\star}\right),\ \ \ \ 0<a<b.
\end{equation}
Now, since $T_2^\star$ is distributed as
$Gamma(m-1,\alpha^{-1})$, we have $2\alpha
T_2^\star\sim\chi^2_{2(m-1)}$. So, we can use the conditions in
\eqref{IML:beta:known} for obtaining the minimum width confidence
interval simply by replacing $m$ and $T_1^\star$ by $m-1$ and
$T_2^\star$, respectively.

\subsubsection*{Confidence interval for $\beta$}
First, we need the following lemma of Doostparast and Balakrishnan (2011).
\begin{lemma}\label{lemma:doost:bala}
Suppose the random variable $U$ is distributed with pdf $g(x;m)$
as in \eqref{pdf:rm:expo} $(m>1)$ and $T$ is a chi-square
random variable with $\nu$ degrees of freedom. If $U$ and $T$ are
independent, then the pdf of $W:=\frac{2U}{T}$ is given by
\begin{equation}\label{pdf:w}
f_W(w;m,\nu)=\frac{1}{w^{1+\nu/2}
\Gamma(m)\Gamma(\nu/2)^2}\int_{0}^{\infty}\left\{-\ln\left(1-e^{-x}\right)\right\}^{m-1}x^{\nu/2}e^{-x(1+\frac{1}{w})}dx.
\end{equation}
\end{lemma}
Since the random variable $R_m$ and $T_2^\star$ are independent
(Arnold {\em et al.}, 1998) and that $2\alpha T_2^\star$ is
distributed as chi-square with $2(m-1)$ degrees of freedom, by
Lemma \ref{lemma:doost:bala}, we can conclude that
\begin{equation}\label{pdf:log:rm:alpha:unknown}
\frac{\ln R_m-\ln \beta}{T_2^\star}\sim f_W(w;m,2m-2).
\end{equation}
So, an equi-tailed $100(1-\gamma)\%$ confidence interval is
given by
\begin{equation}\label{beta:equi:tailed:alphaunknown}
\left(R_m \left[\prod_{i=1}^m
\left(\frac{R_m}{R_i}\right)^{K_i}\right]^{w_{1-\gamma/2}(m,2m-2)},R_m
\left[\prod_{i=1}^m
\left(\frac{R_m}{R_i}\right)^{K_i}\right]^{w_{\gamma/2}(m,2m-2)}\right),
\end{equation}
where $w_\gamma(m,\nu)$ is the $100\gamma$-th percentile of the
density in \eqref{pdf:w}. For some choices of $m$ and $\gamma$,
the values of $w_\gamma(m,2m-2)$ were obtained by Doostparast and
Balakrishnan (2009) and these are presented in Table
\ref{perectiles:pdf:w}. {\small
\begin{table}
  \centering
  \begin{tabular}{c|cccccc}\hline
$m$& 0.01 & 0.025 & 0.05 &  0.95 &0.975 & 0.99\\
\hline
2&0.00068871&0.00200001&0.00462857&1.28171529&1.99110929&3.40093274\\
3&0.00005829&0.00018945&0.00048138&0.19893804&0.28399815&0.42122379\\
4&0.00000748&0.00002678&0.00007396&0.05887585&0.08503410&0.12545761\\
5&0.00000118&0.00000457&0.00001356&0.02116642&0.03157779&0.04776232\\
6&0.00000021&0.00000088&0.00000279&0.00829748&0.01289715&0.02026765\\
7&0.00000004&0.00000018&0.00000061&0.00339234&0.00551395&0.00908049\\
8&0.00000001&0.00000004&0.00000014&0.00141525&0.00240988&0.00415617\\
\hline
  \end{tabular}
  \caption{
  Percentiles of the density in \eqref{pdf:w} for $m=2,\cdots,8$ and $\nu=2m-2$.}\label{perectiles:pdf:w}
\end{table}
}

Restricting to intervals of the form $(aR_m,b R_m)$, where $0<a<b$, the values of $a$ and $b$ which minimize the width in this subclass of intervals subject to the confidence coefficient being $1-\gamma$ can be obtained by
solving the following equations:
\begin{equation}\displaystyle
\left\{
\begin{array}{l}
f_W\left(\frac{a}{T_2^\star};m,2m-2\right)=f_W\left(\frac{b}{T_2^\star};m,2m-2\right),\\ \\
\int_{a/T_2^\star}^{b/T_2^\star}f_W(x;m,2m-2)dx=1-\gamma,
\end{array}  \right.
\end{equation}
where $f_W(w;m,\nu)$ is as  in \eqref{pdf:w}.

\section{Tests of Hypotheses}\label{sec:test}
In this section, we treat tests of hypotheses
concerning the two parameters of the Pareto distribution in \eqref{cdf:par}. To this end, we consider  the following three cases.
\subsection{$\alpha$ known}
If $\alpha$ is known, then $(T_m,R_m)$ is a joint
sufficient statistic for $\beta$. Since $T_m$ is an ancillary statistic (Glick, 1978), $R_m$ is a partially sufficient statistic for $\beta$.  The joint pdf of $({\bf R,K})$ given by \eqref{like:par}
possesses the MLR property in $R_m$. From Theorem 2 of Lehmann (1997, p. 78) and \eqref{pdf:rm:alpha:known}, the UMP test of size $\gamma$ for testing $H_0:\beta\leq \beta_0$ against the
alternative $H_1:\beta>\beta_0$
is
\begin{equation}\label{alpha:known:ump:onesided:1}
\phi({\bf r,k})=\left\{
\begin{array}{ccc}
1,&\ & R_m\geq\beta_0\left(1-\exp\left\{-\frac{1}{2}\chi^2_{2m,\gamma}\right\}\right)^{-1/\alpha},\\
\\
0,&& R_m<\beta_0\left(1-\exp\left\{-\frac{1}{2}\chi^2_{2m,\gamma}\right\}\right)^{-1/\alpha}.
\end{array}
\right.
\end{equation}
By interchanging inequalities throughout, one obtains in an obvious way the solution for the dual problem. Thus,
the UMP test of size $\gamma$ for testing $H_0:\beta\geq \beta_0$ against the
alternative $H_1:\beta<\beta_0$ is
\begin{equation}\label{alpha:known:ump:onesided:2}
\phi({\bf r,k})=\left\{
\begin{array}{ccc}
1,&\ & R_m\leq\beta_0\left(1-\exp\left\{-\frac{1}{2}\chi^2_{2m,1-\gamma}\right\}\right)^{-1/\alpha},\\
\\
0,&& R_m>\beta_0\left(1-\exp\left\{-\frac{1}{2}\chi^2_{2m,1-\gamma}\right\}\right)^{-1/\alpha}.
\end{array}
\right.
\end{equation}
The UMP tests in \eqref{alpha:known:ump:onesided:1} and \eqref{alpha:known:ump:onesided:2} imply that the UMP test for testing $H_0:\beta=\beta_0$ against the alternative $H_1:\beta\neq\beta_0$ does not exist.\\

From \eqref{like:par}, the likelihood ratio statistic is obtained as
$\Lambda=(\beta_0/R_m)^{\alpha T_m}$ for $r_m\geq\beta_0$ and $\Lambda=0$ for $r_m<\beta_0$. Thus, the
critical region of the GLR test of level $\gamma$ is $C=\{({\bf
r,k}):\alpha T_m(\log R_m-\log \beta_0) <c,\ \ or\ \ R_m<\beta_0\}$.
Since $\alpha T_m (\log R_m-\log\beta_0)|(T_m=j)$ has a standard exponential distribution and
$T_m$ is distributed free from the parent distribution (Glick, 1978), we have
the following proposition.
\begin{proposition}
Critical region of the GLR test of level $\gamma$ for testing
$H_0:\beta=\beta_0$ against the alternative
$H_1:\beta\neq\beta_0$ is given by
\begin{equation}\label{alpha:known:glr:twoside}
C=\left\{({\bf r,k}):\left(\frac{\beta_0}{R_m}\right)^{\alpha T_m}<\gamma \ or\
 R_m<\beta_0 \right\}.
\end{equation}
\end{proposition}
As mentioned earlier, the theory of UMP one-sided test can be applied to the problem of obtaining a lower or upper bound. Thus, from \eqref{alpha:known:ump:onesided:1} and \eqref{alpha:known:ump:onesided:2}, $100(1-\gamma)\%$ UMA lower and upper bounds for $\beta$ are given by
\[\left(R_m \left(1-\exp\left\{-\frac{1}{2}\chi^2_{2m,\gamma}\right\}\right)^{1/\alpha},\infty\right)\]
and
\[\left(0,R_m \left(1-\exp\left\{-\frac{1}{2}\chi^2_{2m,1-\gamma}\right\}\right)^{1/\alpha}\right),\]
respectively.
\subsection{$\beta$ known}
If $\beta$ is known, the statistic $T_1^{\star}$ is a complete sufficient statistic, and so all inference can be based on it.
Since $T_1^{\star}$ has a $Gamma(m,\alpha^{-1})$-distribution, and has MLR in $-T_1^{\star}$, the UMP tests of size $\gamma$ for testing $H_0:\alpha\leq \alpha_0$ against the
alternative $H_1:\alpha>\alpha_0$ and $H_0:\alpha\geq \alpha_0$ against the
alternative $H_1:\alpha<\alpha_0$ are
\begin{equation}\label{beta:known:ump:onesided:1}
\phi({\bf r,k})=\left\{
\begin{array}{ccl}
1,&\ & 2\alpha_0 T_1^{\star}\leq \chi^2_{2m,\gamma},\\
\\
0,&& otherwise,
\end{array}
\right.
\end{equation}
and
\begin{equation}\label{beta:known:ump:onesided:1}
\phi({\bf r,k})=\left\{
\begin{array}{ccl}
1,&\ & 2\alpha_0 T_1^{\star}\geq \chi^2_{2m,1-\gamma},\\
\\
0,&& otherwise,
\end{array}
\right.
\end{equation}
respectively. Therefore, the UMP test for testing
$H_0:\alpha=\alpha_0$ against the alternative
$H_1:\alpha\neq\alpha_0$ does not exist. One can easily show
that the critical region of the GLR test of level $\gamma$ for
testing $H_0:\alpha=\alpha_0$ against the alternative
$H_1:\alpha\neq\alpha_0$ is
\begin{equation}\label{beta:known:glr:twoside}
C=\left\{({\bf r,k}):Z_1^m\exp\left\{-\frac{1}{2}Z_1\right\}<c^{\star}\right\},
\end{equation}
where under $H_0$, $Z_1:=2\alpha_0 T_1^\star\sim\chi^2_{2m}$ and
$c^\star$ is chosen such that
\[\gamma=P_{\alpha=\alpha_0}\left(Z_1^m\exp\left\{-\frac{1}{2}Z_1\right\}<c^{\star}\right).\]
Table \ref{table:doost:bala} presents simulated critical values for applying the GLR test, obtained by  Doostparast and Balakrishnan (2011), for some choices of $m$ and $\gamma$.
\begin{table}
  \centering
  \begin{tabular}{c|ccccc}\hline
&  &  &  $m$ &  &  \\ \hline
    $\gamma$& 1 & 2 & 3 & 4 & 5 \\ \hline
    0.01 & 0.0169 & 0.0589 & 0.3139 & 2.2636 & 20.7899 \\
    0.02 & 0.0332 & 0.1136 & 0.6004 & 4.3280 & 39.7496 \\
    0.03 & 0.0490 & 0.1658 & 0.8766 & 6.2658 & 57.6354 \\
    0.04 & 0.0647 & 0.2175 & 1.1370 & 8.1415 & 74.3884 \\
    0.05 & 0.0797 & 0.2664 & 1.3850 & 9.9380 & 90.4894 \\
    0.1 & 0.1517 & 0.4918 & 2.5377 & 18.0175 & 163.4582 \\ \hline
  \end{tabular}
  \caption{Quantiles of $X^m\exp\{-X/2\}$, where $X\sim\chi^2_{2m}$, for
  some choices of $m$ and $\gamma$.}\label{table:doost:bala}
\end{table}

\subsection{Unknown $\beta$ and $\alpha$ }
\subsubsection*{Hypotheses tests for $\alpha$}
There is no UMP test for one-sided hypotheses
on the scale parameter $\beta$.  So, we restrict our attention to smaller classes of tests and seek UMP tests in these subclasses.

The family of densities $\{Par(\beta,\alpha):\beta>0,\ \alpha>0\}$ remains invariant under translations $R_i'=R_i/c$, $0<c<\infty$. Moreover, the hypotheses-testing problem remains invariant under the group of translations, that is, both families of pdfs $\{Par(\beta,\alpha),\ \alpha\geq\alpha_0\}$ and $\{Par(\beta,\alpha),\ \alpha<\alpha_0\}$ remain invariant. On the other hand, the joint sufficient statistic is $(R_m,T_2^\star,T_m)$, which is transformed to
$(R_m/c,T_2^\star,T_m)$. It follows that the class of invariant tests consists of tests
that are functions of $T_2^\star$.  Since $T_2^\star\sim Gamma(m-1,\alpha^{-1})$, the pdf of $T_2^\star$ possesses the MLR property in $-T_2^\star$, and it
therefore follows that a UMP test rejects
$H_0:\alpha\leq\alpha_0$ if $T_2^\star< c$, where $c$ is determined from the size restriction. Hence, we have the following proposition which presents a UMP invariant test for one-sided hypotheses.
\begin{proposition}
To test $H_0: \alpha\leq\alpha_0$ against $H_1:\alpha>\alpha_0$, a UMP invariant test of size $\gamma$ is
\begin{equation}\label{umpi:onesided:1}
\phi({\bf r,k})=\left\{
\begin{array}{ccl}
1,&\ & 2\alpha_0 T_2^{\star}\leq \chi^2_{2m-2,\gamma},\\
\\
0,&& otherwise,
\end{array}
\right.
\end{equation}
and to test $H_0: \alpha\geq\alpha_0$ against $H_1:\alpha<\alpha_0$, a UMP invariant test of size $\gamma$
is
\begin{equation}\label{umpi:onesided:1}
\phi({\bf r,k})=\left\{
\begin{array}{ccl}
1,&\ & 2\alpha_0 T_2^{\star}\geq \chi^2_{2m-2,1-\gamma},\\
\\
0,&& otherwise.
\end{array}
\right.
\end{equation}
\end{proposition}
There is no UMP test for testing $H_0:\alpha=\alpha_0$ against the alternative
$H_1:\alpha\neq\alpha_0$. We therefore use the GLR procedure for this testing problem.
The likelihood ratio function is given by
\begin{equation}
\Lambda=\left(\frac{Z_2}{2m}\right)^m\exp\{-m(Z_2/2m-1)\},
\end{equation}
where $Z_2=2\alpha_0 T_2^\star$ which, under $H_0$, is distributed
as chi-square with $2(m-1)$ degrees of freedom.
Hence, the critical region of GLR test at level $\gamma$ is
\begin{equation}\label{critical:glr:for:alpha:two:sided}
C=\{({\bf r,k}):y^m\exp(-y/2)<a\},
\end{equation}
where $a$ is chosen such that $\gamma=P(Z_2^m\exp\{-Z_2/2\}<a)$
and $Z_2\sim \chi^2_{2(m-1)}.$ Table \ref{table:doost:bala}
presents critical values for applying the GLR test for some
choices of $m$ and $\gamma$.

\subsubsection*{Hypotheses tests for $\beta$}
In the case of unknown $\alpha$, finding a UMP test  for one- and
two-sided hypotheses remains as an open problem. However,
$\displaystyle\hat{\alpha}_{M,0}=\frac{m}{\sum_{i=1}^m K_i(\log R_i-\log\beta_0)}$ is the
maximum likelihood estimator of $\alpha$ under
$H_0:\beta=\beta_0$. This fact and \eqref{like:par} yield the
likelihood ratio statistic, for testing $H_0:\beta=\beta_0$
against the alternative $H_1:\beta\neq\beta_0$, as
\begin{equation}\label{glr:statistics:for:beta:2}
\Lambda=\left\{\begin{array}{ll}\left(\frac{T_2^\star}{T_0^\star}\right)^m & \mbox{for}\ \ r_m\geq\beta_0\\
0 & \mbox{for}\ \ r_m<\beta_0\end{array}, \right.
\end{equation}
where $T_0^\star=\sum_{i=1}^m K_i (\log R_i-\log\beta_0)$. But,
$$\frac{T_2^\star}{T_0^\star}=1-\frac{T_m(\log R_m-\log\beta_0)}{m\hat{\alpha}_{M,0}}.$$
Therefore, the
critical region of the GLR test of level $\gamma$ for testing
$H_0:\beta=\beta_0$ against the alternative $H_1:\beta\neq\beta_0$ is given by
\begin{equation}\label{glr:twoside:for:beta}
C=\left\{({\bf r,k}):T_m(\log R_m-\log\beta_0) >\hat{\alpha}_{M,0}C^\star\ or\
 R_m<\beta_0 \right\},
\end{equation}
where $C^\star$ is obtained from the size
restriction
\begin{equation}\label{cstar:glr:twoside:for:beta}
\gamma=P_{\beta_0}\left(T_m(\log R_m-\log\beta_0) >\hat{\alpha}_{M,0}
C^\star\right).
\end{equation}
An explicit closed-form expression for $C^\star$ in
\eqref{glr:twoside:for:beta} does not seem to be possible. But,
one can use the following expression for computational purposes:
\begin{eqnarray}
\gamma&=&P_{\beta_0}\left(T_m(\log R_m-\log\beta_0)
>\hat{\alpha}_{M,0}
C^\star\right)\nonumber\\
&=&\sum_{j=m}^{\infty}P_{\beta_0}\left(T_m(\log R_m-\log\beta_0)
>\hat{\alpha}_{M,0}
C^\star|T_m=j\right)P(T_m=j)\nonumber\\
&=&\sum_{j=m}^{\infty}P_{\beta_0}\left(\frac{\log
R_m-\log\beta_0}{\sum_{i=1}^m K_i(\log
R_i-\log\beta_0)}>\frac{C^\star}{m
j}|T_m=j\right)\frac{1}{j!}\left[\begin{array}{c}j-1\\m-1\end{array}\right].
\end{eqnarray}

\section{Numerical Example}
Dyer (1981) reported annual wage data (in multiplies of 100 U.S.
dollars) of a random sample of 30 production-line workers in a
large industrial firm, as presented in Table \ref{wage:data}.
\begin{table}
\caption{Annual wage data (in multiplies of 100 U.S. dollars).}
\label{wage:data}       
\begin{tabular}{cccccccccc}
\hline
112&154&119&108&112&156&123&103&115&107\\
125&119&128&132&107&151&103&104&116&140\\
108&105&158&104&119&111&101&157&112&115\\
 \hline
\end{tabular}
\centering
\end{table}
He determined that Pareto distribution provided an adequate fit
for data. Assuming inverse sampling scheme with $m=3$, the
corresponding record data are presented in Table
\ref{wage:data:record}.
\begin{table}
\caption{Record data arising from annual wage data with $m=3$.}
\label{wage:data:record}       
\begin{tabular}{cccc}
\hline
$i$& 1&2&3\\
$R_i$&112&108&103\\
$K_i$&3&4&1 \\
\hline
\end{tabular}
\centering
\end{table}
From \eqref{ml:beta} and \eqref{ml:alpha}, the MLE of $\beta$ and
$\alpha$ on the basis of record data are obtained to be
$\hat{\beta}_M=103$ and $\hat{\alpha}_M=6.804$,
respectively. From \eqref{beta:equi:tailed:alphaunknown}, an
equi-tailed $95\%$ confidence interval for $\beta$ is obtained to
be \[\left(90.877, 102.991\right).\] Similarly, a
minimum-width $95\%$ confidence interval for $\alpha$ in the
class \eqref{class2:int:alpha} is obtained as
\[\left(0.096, 10.807\right).\]
For testing $H_0:\alpha=6$ against the alternative
$H_1:\alpha\neq 6$, we have $Z_2=2\alpha_0 T_2^\star=5.291$,
which gives
\[Z_2^m\exp\left\{-\frac{Z_2}{2}\right\}=10.512.\]
From Table \ref{table:doost:bala}, we have $a=0.2664$. Therefore,
\eqref{critical:glr:for:alpha:two:sided} implies that $H_0$ is not
rejected. Since we can not find $C^\star$ in
\eqref{glr:twoside:for:beta}, we can not test the hypothesis
$H_0:\beta=\beta_0$ against the alternative $H_1:\beta\neq
\beta_0$. In this case, one may conduct a simulation study and
calculate the percentile of $\Lambda$ in
\eqref{glr:statistics:for:beta:2} by specifying $\beta_0$, and then carry out a likelihood-ratio test.

\section*{Acknowledgements}
The authors are grateful to anonymous referees and the associate editor for their useful suggestions and comments on
an earlier version of this manuscript, which resulted in this improved version.
\section*{Bibliography}

\begin{enumerate}
\item Arnold, B. C. (1983). {\it Pareto Distributions}, International Co-operative Publishing House, Fairland, MD.
\item Arnold, B. C., Balakrishnan, N. and Nagaraja, H. N. (1992).
{\it A First Course in Order Statistics}, John Wiley \& Sons, New York.
\item Arnold, B. C., Balakrishnan, N. and Nagaraja, H. N. (1998).
{\it Records}, John Wiley \& Sons, New York.
\item Bardlow, E. T. and Park, Y. (2007). Bayesian estimation of bid sequences in internet auction using a generalized record-breaking model,
{\it Marketing Science}, {\bf 26}, 218--229.
\item Benestad, R. E. (2003). How often can we expect a record event?, {\it Climate Research}, {\bf 25}, 3--13.
\item de Haan, L., de Vries, C. G. and Zhou, C. (2009). The expected payoff to Internet auctions, {\it Extremes}, {\bf 12}, 219--238.
\item  Doostparast, M. (2009). A note on estimation based on record data, {\it Metrika}, {\bf 69}, 69--80.
\item  Doostparast, M. and Balakrishnan, N. (2011). Optimal record-based statistical procedures for the two-parameter exponential distribution, {\it Journal of Statistical Computation and Simulation}, DOI: 10.1080/00949655.2010.513979.
\item  Dyer, D. (1981). Structural probability bounds for the strong Pareto laws, {\it The Canadian Journal of Statistics},
{\bf 9}, 71--77.
\item Glick, N. (1978). Breaking records and breaking boards,
{\it American Mathematical Monthly}, {\bf 85},
2--26.
\item Hsieh, P. (2004). A data-analytic method for forecasting next record catastrophe loss, {\it Journal of Risk and Insurance}, {\bf 71}, 309--322.
\item Johnson, N. L., Kotz, S. and Balakrishnan, N. (1994). {\it Continuous Univariate Distribution- Vol. 1,} Second edition, John Wiley \& Sons, New York.
\item Krug, J. and Jain, K. (2005). Breaking records in the evolutionary race,
{\it Physica A}, \textbf{53}, 1--9.
\item Kuper, G. H. and Sterken, E. (2003). Endurance in speed skating: The development of world records, {\it European Journal of Operational Research,} {\bf 148}, 293--301.
\item Lehmann, E. L. (1997).  {\it Testing Statistical Hypotheses}, Second edition, Springer-Verlag, New York.
\item Lehmann, E. L. and Romano, J. P. (2005).  {\it Testing Statistical Hypotheses}, Third edition, Springer-Verlag, New York.
\item  Pachares, J. (1961). Tables for unbiased tests on the variance of a normal population, {\it Annals of Mathematical Statistics},
{\bf 32}, 84--87.
\item Pfeifer D. (1997). A statistical model to analyse natural catastrophe claims by mean of record values, In: {\it Proceedings of the 28 International ASTIN Colloquium, Cairns, Australia}, August 10-12, 1997, The Institute of Actuaries of Australia.
\item Samaniego, F. J. and Whitaker, L. R. (1986). On estimating population
 characteristics from record-breaking observations, I. Parametric
results, {\it Naval Research Logistics Quarterly}, \textbf{33},
531--543.
\item Samaniego,
F. J. and Whitaker, L. R. (1988). On estimating population
characteristics from record-breaking observations, II.
Nonparametric Results, {\it Naval Research Logistics Quarterly},
\textbf{35}, 221--236.
\item Sibuya, M. and Nishimura, K. (1997). Prediction of record-breakings, {\it Statistica Sinica},
\textbf{7}, 893--906.
\item Yang, T. Y. and Lee, J. C. (2007). Bayesian nearest-neighbor analysis via
record value statistics and nonhomogeneous spatial Poisson
processes, {\it Computational Statistics \& Data Analysis}, {\bf
51}, 4438--4449.
%
%
\end{enumerate}

\end{document}